\documentclass[11pt]{amsart}%
\usepackage{amscd,amssymb,verbatim}
\usepackage{amsmath}
\usepackage{amsfonts}
\usepackage{amssymb}
\usepackage{graphicx}%
\providecommand{\U}[1]{\protect\rule{.1in}{.1in}}
\newtheorem{thm}{Theorem}

\newtheorem{lem}[thm]{Lemma}
\newtheorem{prop}[thm]{Proposition}

\begin{document}
\title[Nonlinear order isomorphisms on Banach lattices]{Persistence of Banach lattices under nonlinear order isomorphisms}

\begin{abstract}
Ordered vector spaces $E$ and $F$ are said to be order isomorphic if there is
a (not necessarily linear) bijection $T: E\to F$ such that $x\geq y$ if and
only if $Tx \geq Ty$ for all $x,y \in E$. We investigate some situations under
which an order isomorphism between two Banach lattices implies the persistence
of some linear lattice structure. For instance, it is shown that if a Banach
lattice $E$ is order isomorphic to $C(K)$ for some compact Hausdorff space
$K$, then $E$ is (linearly) isomorphic to $C(K)$ as a Banach lattice. Similar
results hold for Banach lattices order isomorphic to $c_{0}$, and for Banach
lattices that contain a closed sublattice order isomorphic to $c_{0}$.

\end{abstract}
\author{Denny H.\ Leung}
\address{Department of Mathematics, National University of Singapore, Singapore 119076}
\email{matlhh@nus.edu.sg}
\author{Wee-Kee Tang}
\address{Division of Mathematical Sciences, Nanyang Technological University, Singapore
637371 }
\email{WeeKeeTang@ntu.edu.sg }
\thanks{Research of the first author is partially supported by AcRF project
no.\ R-146-000-157-112. Research of the second author is partially supported
by AcRF project no.\ RG26/14.}
\maketitle


Two ordered vector spaces $E$ and $F$ are said to be \emph{order isomorphic}
if there is a (not necessarily linear) bijection $T:E\to F$ so that $x\geq y$
if and only if $Tx\geq Ty$ for all $x,y\in E$. In this case, we call $T$ an
\emph{order isomorphism}. When $E$ and $F$ are Banach lattices, there is the
well studied notion of (vector) lattice isomorphism: $E$ and $F$ are
\emph{lattice isomorphic} if there is a linear bijection $T:E\to F$ such that
$T|x| = |Tx|$ for all $x\in E$. This is equivalent to the existence of a
linear order isomorphism from $E$ onto $F$. It is well known that a lattice
isomorphism $T$ between Banach lattices must also be an isomorphism between
the underlying Banach spaces; that is, both $T$ and $T^{-1}$ must be bounded.
It is easy to see that, in general, two Banach lattices that are order
isomorphic need not be lattice isomorphic. Indeed, for any measure space
$(\Omega,\Sigma,\mu)$ and any $1< p <\infty$, the map $f\mapsto|f|^{p}%
\operatorname{sgn} f$ is an order isomorphism from $L^{p}(\Omega,\Sigma,\mu)$
onto $L^{1}(\Omega,\Sigma,\mu)$. However, $L^{p}(\Omega,\Sigma,\mu)$ and
$L^{1}(\Omega,\Sigma,\mu)$ are not lattice isomorphic unless they are finite
dimensional. In contrast to the situation for $L^{p}$ spaces, it is shown in
this paper that some vector lattice properties pertaining to $AM$-(or abstract
$M$-) spaces persist under order isomorphisms. For the definition of
$AM$-spaces, as well as for general background with regard to the theory of
Banach lattices, we refer the reader to \cite{MN, S}. By the well known
Kakutani's representation theorem, a Banach lattice is an $AM$-space if and
only if it is isometrically lattice isomorphic to a closed sublattice of
$C(K)$ for some compact Hausdorff space $K$; see, e.g., \cite[Theorem
1.b.6]{LT}. Our first result is quite simple. If $u$ is a positive element in
a Banach lattice $E$, let $E_{u}$ be the closed ideal in $E$ generated by
$u$,
\[
E_{u} = \{x\in E: |x| \leq nu \text{ for some $n\in{\mathbb{N}}$}\}.
\]
$u$ is an \emph{order unit} of $E$ if $E_{u}= E$. It is a standard fact that
if $E$ has an order unit, then $E$ is lattice isomorphic to $C(K)$ for some
compact Hausdorff space $K$; see \cite[Proposition II.7.2 and Corollary 1 to
Theorem II.7.4]{S}.

\begin{thm}
\label{t0} Let $E$ be a Banach lattice. If $E$ is order isomorphic to $C(K)$
for some compact Hausdorff space $K$, then $E$ is lattice isomorphic to $C(K)$.
\end{thm}

\begin{proof}
Let $T:C(K)\to E$ be an order isomorphism. We may assume that $T0 = 0$. For
any $n\in{\mathbb{N}}$, we use the same symbol $n$ denote the constant
function on $K$ with value $n$. Then $C(K)_{+} = \cup_{n}[0,n]$. Hence $E_{+}
= \cup_{n}[0,x_{n}]$, where $x_{n} = Tn$. By the Baire Category Theorem, there
exists $n_{0}$ such that $[0,x_{n_{0}}]$ contains nonempty interior. Thus
$E_{+}$ has an interior point $u$. By \cite[Corollary 1.2.14]{MN}, $u$ is an
order unit of $E$. It follows that $E$ is lattice isomorphic to $C(L)$ for
some compact Hausdorff space $L$. In this case, $C(K)$ and $C(L)$ are
nonlinearly order isomorphic. By \cite[Proposition 3]{CS}, $K$ and $L$ are
homeomorphic. Thus $C(K)$ and $C(L)$ are lattice isomorphic. Since $E$ is
lattice isomorphic to $C(L)$, the proof is complete.
\end{proof}

We do not know if a Banach lattice that is order isomorphic to an $AM$-space
must be lattice isomorphic to an $AM$-space. In this direction, there is a
useful characterization of $AM$-spaces due to Cartwright and Lotz; see
\cite{CL} and \cite[Theorem 2.1.12]{MN}. A subset $A$ in an ordered vector
space $E$ is \emph{order bounded} if there are $u, v\in E$ such that $u \leq
x\leq v$ for all $x\in A$. A sequence $(x_{n})$ in a vector lattice is
\emph{disjoint} if $|x_{m}| \wedge|x_{n}| =0$ whenever $m \neq n$.

\begin{thm}
\emph{(Cartwright and Lotz)} \label{t1} A Banach lattice $E$ is lattice
isomorphic to an $AM$-space if and only if every disjoint norm null sequence
in $E$ is order bounded in $E^{\prime\prime}$.
\end{thm}

With the help of this theorem, we offer a partial solution to the problem
raised above. A subspace $F$ of a Banach lattice $E$ is an \emph{(order)
ideal} if $y\in F$ for all $y\in E$ such that $|y| \leq|x|$ for some $x\in F$.
By \cite[Proposition 2.1.9]{MN}, every closed ideal in $C(K)$ has the form
\[
I = \{f\in C(K): f =0 \text{ on $K_{0}$}\} \text{ for some closed subset
$K_{0}$ of $K$}.
\]

\begin{prop}
\label{p3} Let $E$ be a Banach lattice. If $E$ is order isomorphic to a closed
ideal of some space $C(K)$, where $C(K)$ is separable, then $E$ is lattice
isomorphic to an $AM$-space.
\end{prop}

\begin{proof}
Let $T:E\to I$ be an order isomorphism, where $I$ is a closed ideal in $C(K)$,
with $C(K)$ separable. We may assume that $T0 = 0$. Since $C(K)$ is separable,
$K$ is metrizable. Let $d$ be a metric on $K$ generating the given topology.
There is a closed set $K_{0}$ in $K$ so that $I$ consists of all functions in
$C(K$) that vanish on $K_{0}$. By Theorem \ref{t1}, it suffices to show that
every disjoint norm null sequence in $E$ is order bounded in $E$. Let
$(x_{n})$ be a disjoint null sequence in $E$. Define $f_{n} = T|x_{n}|$ for
all $n$. Then $(f_{n})$ is a disjoint nonnegative sequence in $I$. If
$(f_{n})$ is not norm bounded, there is a subsequence $(f_{n_{k}})$ such that
$\|x_{n_{k}}\| \leq1/2^{k}$ and $\|f_{n_{k}}\| > k$ for all $k$. The sum $x =
\sum|x_{n_{k}}|$ converges in $E$. Clearly $Tx \geq f_{n_{k}} \geq0$ for all
$k$. This implies that $\|Tx\| > k$ for all $k$, which is absurd. Therefore,
there exists $c_{0}$ such that $c_{0} > \|f_{n}\|$ for all $n$.

\medskip

\noindent\underline{Claim}. Let $c_{k} = \sup\{f_{n}(t): d(t,K_{0}) \leq1/k,
n\in{\mathbb{N}}\}$. Then $(c_{k})$ is a nonincreasing null sequence.

\medskip

Clearly $(c_{k})$ is a nonincreasing sequence. If $(c_{k})$ is not a null
sequence, there exists $\varepsilon>0$ such that $c_{k} > \varepsilon$ for all
$k$. By uniform continuity of $f_{n}$, for each $n$, $\lim_{k}\sup\{f_{n}(t):
d(t,K_{0})\leq1/k\} =0$. Thus, there exist $n_{1}< n_{2}< \cdots$ and
$(t_{i})$ in $K$, $d(t_{i},K_{0})\to0$, such that $f_{n_{i}}(t_{i}) >
\varepsilon$ for all $i$. By taking a further subsequence if necessary, we may
also assume that $\|x_{n_{i}}\|\leq1/2^{i}$ for all $i$. Now $x =
\sum|x_{n_{i}}|$ converges in $E$ and $Tx \geq f_{n_{i}}$ for all $i$. Then
$Tx(t_{i}) \geq f_{n_{i}}(t_{i}) > \varepsilon$ for all $i$. Since
$d(t_{i},K_{0})\to0$, this contradicts the fact that $Tx\in I$.

\medskip

By the Claim, there exists a continuous function $g$ on $[0,\infty)$ such that
$g(0) = 0$, $g(s) \geq c_{k}$ if $\frac{1}{k+1} \leq s < \frac{1}{k}$, where
we take $1/0 = \infty$. Define $f:K \to{\mathbb{R}}$ by $f(t) = g(d(t,K_{0}%
))$. Then $f\in C(K)$ and $f=0$ on $K_{0}$. Hence $f\in I$. For any $n$, if
$d(t,K_{0}) = 0$, then $t\in K_{0}$ and hence $f_{n}(t) =0 \leq f(t)$. On the
other hand, if $\frac{1}{k+1} \leq d(t,K_{0}) < \frac{1}{k}$, then $f_{n}(t)
\leq c_{k} \leq f(t)$. Thus $f \geq f_{n}$ for all $n\in{\mathbb{N}}$. Then
$T^{-1}f \geq|x_{n}|$ for all $n\in{\mathbb{N}}$. Therefore, $(x_{n})$ is
order bounded in $E$, as desired.
\end{proof}

Now we can show that the Banach lattice $c_{0}$ is stable under nonlinear
order isomorphisms.

\begin{thm}
\label{t4} Let $E$ be a Banach lattice. The following are equivalent.

\begin{enumerate}
\item $E$ is lattice isomorphic to $c_{0}$.

\item $E$ is order isomorphic to $c_{0}$.

\item $E$ is order isomorphic to an infinite dimensional closed sublattice of
$c_{0}$.
\end{enumerate}
\end{thm}

\begin{proof}
The implications (a) $\implies$ (b) $\implies$ (c) are immediate. By
\cite[Corollary 5.3]{RT}, every infinite dimensional closed sublattice of
$c_{0}$ is lattice isomorphic to $c_{0}$. The implication (c) $\implies$ (b)
follows. Now assume that $E$ is order isomorphic to $c_{0}$. Let $T: c_{0}\to
E$ be an order isomorphism such that $T0 = 0$. Denote by $(e_{n})$ the unit
vector basis of $c_{0}$ and let $x_{n} = Te_{n}$ for each $n$. If $m\neq n$,
$0 = T(e_{m}\wedge e_{n}) = x_{m}\wedge x_{n}$. That is, $(x_{n})$ is a
disjoint positive sequence in $E$. Also, since $[0,e_{n}]$ is a totally
ordered set, so is $[0,x_{n}]$. It follows that $[0,x_{n}] = \{cx_{n}: 0\leq
c\leq1\}$.

\medskip

\noindent\underline{Claim}. For each $n\in{\mathbb{N}}$ and any $a\geq0$,
there exists $b\geq0$ such that $T(ae_{n}) = bx_{n}$.

\medskip

Let $a \geq0$ be given and define $b= \sup\{c\geq0: cx_{n} \leq T(ae_{n})\}$.
Obviously, the set on the right contains $0$ and hence is nonempty. Also
$cx_{n} \leq T(ae_{n})$ implies that $|c|\|x_{n}\| \leq\|T(ae_{n})\|$. Since
$x_{n}\neq0$, it follows that $b <\infty$. There exist $c_{k} \geq0$ such that
$c_{k}x_{n} \leq T(ae_{n})$ and $c_{k} \to b$. Since $E_{+}$ is a closed set,
$bx_{n} \leq T(ae_{n})$. Let $x = T(ae_{n}) -bx_{n}$. Then $x\geq0$. Thus
$T^{-1}x = \sum a_{m}e_{m} = \bigvee a_{m}e_{m}\geq0$ in $c_{0}$. If $m\neq
n$,
\[
0 = T(ae_{n}\wedge e_{m}) = T(ae_{n}) \wedge x_{m} \geq x\wedge x_{m}\geq0.
\]
Thus $x\wedge x_{m} = 0$ if $m\neq n$. On the other hand, since $x\wedge
x_{n}\in[0,x_{n}]$, there exists $0 \leq c\leq1$ such that $x\wedge x_{n} =
cx_{n}$. Then
\[
T(ae_{n})-bx_{n} = x\geq x\wedge x_{n} = cx_{n}%
\]
and hence $T(ae_{n}) \geq(b+c)x_{n}$. By definition of $b$, $c = 0$. Hence $x
\wedge x_{n} = 0$. Therefore, $T^{-1}x \wedge e_{i} = T^{-1}(x \wedge x_{i})
=0$ for all $i$. Clearly, this means that $T^{-1}x=0$ and hence $x = 0$. So we
have shown that $T(ae_{n}) = bx_{n}$, as desired. This completes the proof of
the Claim.

\medskip

Let $x$ be any positive element in $E$. Then $T^{-1}x = \bigvee a_{n}e_{n}$
for some nonnegative sequence $(a_{n}) \in c_{0}$. Thus $x = \bigvee
T(a_{n}e_{n})$. By the Claim, $x = \bigvee b_{n}x_{n}$ for some nonnegative
scalars $b_{n}$. If $\bigvee b_{n}x_{n} = \bigvee b_{n}^{\prime}x_{n}$, where
$b_{n},b_{n}^{\prime}\geq0$ and both suprema exist, then using the
distributivity of the lattice operations, it is easy to see that $b_{n} =
b_{n}^{\prime}$ for all $n$.

Now we show that for any $x = \bigvee b_{n}x_{n}$ as described above,
$\lim\|b_{n}x_{n}\| = 0$. Otherwise, there exist $\varepsilon>0$ and an
infinite subset $I$ of ${\mathbb{N}}$ so that $\|b_{n}x_{n}\| \geq\varepsilon$
for all $n\in I$. For each $k\in{\mathbb{N}}$, $T^{-1}(kb_{n}x_{n}) =
\bigvee_{m}a_{k,m}e_{m}$. If $i\neq n$,
\[
0 = T^{-1}(kb_{n}x_{n}\wedge x_{i})= T^{-1}(kb_{n}x_{n}) \wedge T^{-1}x_{i} =
(\bigvee_{m}a_{k,m}e_{m}) \wedge e_{i} = (a_{k,i}\wedge1)e_{i}.
\]
Thus $a_{k,i} = 0$ if $i\neq n$. Hence $T^{-1}(kb_{n}x_{n}) = a_{k,n}e_{n}$.
Then $T^{-1}(kx) = \bigvee a_{k,n}e_{n}$. In particular, $\lim_{n}a_{k,n}= 0$
for all $k$. Choose $n_{1}< n_{2}<\cdots$ in $I$ so that $\lim_{k}a_{k,n_{k}}
= 0$. We have $z = \bigvee a_{k,n_{k}}e_{n_{k}} \in c_{0}$ and $z\geq
T^{-1}(kb_{n_{k}}x_{n_{k}})$ for all $k$. Hence $Tz \geq kb_{n_{k}}x_{n_{k}}$
for all $k$. But then $\|Tz\| \geq\|kb_{n_{k}}x_{n_{k}}\| \to\infty$, which is
impossible. This proves that $\lim\|b_{n}x_{n}\| = 0$.

To recap, we have shown that if $x\in E_{+}$, then $x$ has a unique
representation $x = \bigvee b_{n}x_{n}$, where $b_{n}$ are nonnegative scalars
so that $\lim\|b_{n}x_{n}\| = 0$. Note that $c_{0}$ is a closed ideal in the
space $C({\mathbb{N}}^{*})$, where ${\mathbb{N}}^{*}$ is the $1$-point
compactification of ${\mathbb{N}}$, and that $C({\mathbb{N}}^{*})$ is
separable. By Proposition \ref{p3}, $E$ is lattice isomorphic to an
$AM$-space. Consider the linear map $S:c_{0}\to E$ given by $S(b_{n}) = \sum
b_{n}x_{n}/\|x_{n}\|$. Note that if $(b_{n}) \in c_{0}$, then $\sum b_{n}%
x_{n}/\|x_{n}\|$ converges in $E$ since $E$ is an $AM$-space. Since $(x_{n})$
is a disjoint sequence, $S$ is an injection. If $x\in E_{+}$, then $x =
\bigvee b_{n}x_{n}$, where $b_{n}$ are nonnegative scalars so that
$\lim\|b_{n}x_{n}\| = 0$. Thus $S(b_{n}\|x_{n}\|) = \sum b_{n}x_{n} = \bigvee
b_{n}x_{n} = x$. Hence the range of $S$ contains $E_{+}$. It follows that $S$
is onto. It is clear that $S(b_{n}) \geq0$ if $(b_{n})\geq0$. Since $S$ is a
bijection as well, $S$ is an order isomorphism. Hence it is a linear order
isomorphism and thus a lattice isomorphism.
\end{proof}

In view of Theorems \ref{t0} and \ref{t4}, and the example of $L^{p}$ spaces
mentioned in the introduction, it seems reasonable to ask the following question.

\bigskip

\noindent\textbf{Problem}. Suppose that $E$ is a Banach lattice so that any
Banach lattice that is order isomorphic to $E$ is (linearly) lattice
isomorphic to $E$. Must $E$ be an $AM$-space?

\bigskip

We can offer a partial solution to the problem. An element $e \geq0$ in a
Banach lattice is an \emph{atom} if the ideal generated by $e$ is one
dimensional. A Banach lattice is \emph{atomic} if there is a maximal
orthogonal set consisting of atoms. Let $E$ be an atomic Banach lattice and
let $(e_{\gamma})_{\gamma\in\Gamma}$ be a maximal orthogonal set consisting of
normalized atoms. Any element $x\in E$ has a unique representation
\begin{equation}
\label{rep}x= \bigvee_{\gamma\in\Gamma_{1}}a_{\gamma}e_{\gamma}-
\bigvee_{\gamma\in\Gamma_{2}}a_{\gamma}e_{\gamma},
\end{equation}
where $\Gamma_{1}$ and $\Gamma_{2}$ are disjoint subsets of $\Gamma$ and $0 <
a_{\gamma}\in{\mathbb{R}}$ for all $\gamma\in\Gamma_{1}\cup\Gamma_{2}$. See,
e.g., \cite[Exercise II.7]{S}. For $1 < p <\infty$, the \emph{$p$%
-convexification} of $E$, denoted by $E^{(p)}$, as defined on p.53 in
\cite{LT}, may be presented as follows. $E^{(p)}$ is the set of all real
sequences $(a_{\gamma})_{\gamma\in\Gamma}$ such that $\bigvee|a_{\gamma}%
|^{p}e_{\gamma}\in E$, endowed with the norm $|||(a_{\gamma})||| =
\|\bigvee|a_{\gamma}|^{p}e_{\gamma}\|^{1/p}$. $E^{(p)}$ is a Banach lattice
(in the pointwise order). For each $\gamma\in\Gamma$, let $u_{\gamma}=
(a_{\xi})_{\xi\in\Gamma}$ with $a_{\xi}= 1$ if $\xi= \gamma$ and $a_{\xi}= 0$
otherwise. Then $(u_{\gamma})_{\gamma\in\Gamma}$ is a maximal orthogonal set
in $E^{(p)}$ consisting of normalized atoms. The map $T:E^{(p)}\to E$,
\[
T(a_{\gamma}) = \bigvee_{\{\gamma:a_{\gamma}\geq0\}} | a_{\gamma}|^{p}e_{n} -
\bigvee_{\{\gamma:a_{\gamma}< 0\}} |a_{\gamma}|^{p}e_{n}%
\]
is a nonlinear order isomorphism. The norm on a Banach lattice $X$ is said to
be \emph{weakly Fatou} \cite[Definition 2.4.18]{MN} if there is a constant $K
< \infty$ so that if $0 \leq x_{\tau}\uparrow x$, then $\|x\| \leq K\sup
_{\tau}\|x_{\tau}\|$.

\begin{thm}
Let $E$ be an atomic Banach space and let $(e_{\gamma})_{\gamma\in\Gamma}$ be
a maximal orthogonal set consisting of normalized atoms. Suppose that any
Banach lattice $F$ that is (nonlinearly) order isomorphic to $E$ is (linearly)
lattice isomorphic to $E$. Then the closed sublattice generated by
$(e_{\gamma})_{\gamma\in\Gamma}$ in $E$ is lattice isomorphic to $c_{0}%
(\Gamma)$. Furthermore, if the norm on $E$ is weakly Fatou, then $E$ is
lattice isomorphic to a closed sublattice of $\ell^{\infty}(\Gamma)$.
\end{thm}

\begin{proof}
Let $F$ be the $2$-convexification of $E$ and let $(u_{\gamma})$ be the
maximal orthogonal set of normalized atoms in $F$ as described above.. Since
$F$ is order isomorphic to $E$, it is lattice isomorphic to $E$ by the
assumption. Let $T:E\rightarrow F$ be a lattice isomorphism. For each $\gamma
$, $Te_{\gamma}$ is a nonzero positive element in $F$ and $[0,Te_{\gamma
}]=T[0,e_{\gamma}]$ lies within a $1$-dimensional subspace. Hence there exist
$\pi(\gamma)\in\Gamma$ and $c_{\gamma}>0$ such that $Te_{\gamma}=c_{\gamma
}u_{\pi(\gamma)}$. Since $T$ is a lattice isomorphism, $\pi:\Gamma
\rightarrow\Gamma$ is a permutation on $\Gamma$ and $0<\inf c_{\gamma}\leq\sup
c_{\gamma}<\infty$. For any finite subset $I$ of $\Gamma$, we have
\begin{align*}
\frac{1}{\sup c_{\gamma}}\Vert\bigvee_{\gamma\in I}e_{\gamma}\Vert &
\leq\Vert\bigvee_{\gamma\in I}\frac{1}{c_{\gamma}}e_{\gamma}\Vert=\Vert
T^{-1}\bigvee_{\gamma\in I}u_{\pi(\gamma)}\Vert\\
&  \leq\Vert T^{-1}\Vert\cdot|||\bigvee_{\gamma\in I}u_{\pi(\gamma)}|||=\Vert
T^{-1}\Vert\Vert\bigvee_{\gamma\in I}e_{\pi(\gamma)}\Vert^{1/2}.
\end{align*}
Let $C=\sup c_{\gamma}\Vert T^{-1}\Vert$. For any $m\in{\mathbb{N}}$, let
\[
\mu_{m}=\sup\{\Vert\bigvee_{n\in I}e_{n}\Vert:\#I=m\}.
\]
Clearly, $\mu_{m}<\infty$. Let $I$ be such that $\#I=m$ and $\Vert
\bigvee_{n\in I}e_{n}\Vert\geq\mu_{m}/2$. Then
\[
\mu_{m}\geq\Vert\bigvee_{\gamma\in\pi(I)}e_{\gamma}\Vert\geq\frac{1}{C^{2}%
}\Vert\bigvee_{\gamma\in I}e_{\gamma}\Vert^{2}\geq\frac{\mu_{m}^{2}}{4C^{2}}.
\]
Therefore, $\mu_{m}\leq4C^{2}$.

Let $G$ be the closed sublattice of $E$ generated by $(e_{\gamma})_{\gamma
\in\Gamma}$. Since $(e_{\gamma})_{\gamma\in\Gamma}$ is a disjoint set, $G$ is
the same as the closed subspace generated by $(e_{\gamma})_{\gamma\in\Gamma}$.
Clearly, $\sum a_{\gamma}e_{\gamma}\in G$ implies that $(a_{\gamma}) \in
c_{0}(\Gamma)$. Conversely, suppose that $(a_{\gamma})\in c_{0}(\Gamma)$. For
any $\varepsilon> 0$, there exists a finite subset $I$ of $\Gamma$ such that
$|a_{\gamma}| \leq\varepsilon$ for all $\gamma\notin I$. If $J$ is a finite
subset of $\Gamma$ disjoint from $I$, then
\[
\|\sum_{\gamma\in J}a_{\gamma}e_{\gamma}\| \leq\max_{\gamma\in J}|a_{\gamma
}|\|\sum_{\gamma\in J}e_{\gamma}\| \leq\varepsilon\cdot4C^{2}.
\]
Thus, $\sum a_{\gamma}e_{\gamma}$ converges in $G$ if $(a_{\gamma}) \in
c_{0}(\Gamma)$. It is now clear that the map $S: c_{0}(\Gamma) \to G$ defined
by $S(a_{\gamma}) = \sum a _{\gamma}e_{\gamma}$ is a lattice isomorphism.

Finally, suppose that the norm on $E$ is weakly Fatou with constant $K$. By
the discussion preceding the theorem, each $x\in E$ has a unique
representation (\ref{rep}). Clearly, for $\gamma\in\Gamma_{1}\cup\Gamma_{2}$,
$|a_{\gamma}| = \|a_{\gamma}e_{\gamma}\| \leq\|x\|$. Define
\[
x(\gamma) =
\begin{cases}
a_{\gamma} & \text{if $\gamma\in\Gamma_{1}$},\\
-a_{\gamma} & \text{if $\gamma\in\Gamma_{2}$},\\
0 & \text{otherwise}.
\end{cases}
\]
Then $(x(\gamma)) \in\ell^{\infty}(\Gamma)$ and $\|(x(\gamma))\|_{\infty}%
\leq\|x\|$. On the other hand $\bigvee_{\gamma\in I} |x(\gamma)|e_{\gamma
}\uparrow|x|$, where $I$ runs through the directed set of all finite subsets
of $\Gamma$. By assumption,
\begin{align*}
\|x\|  &  = \| |x| \| \leq K \sup_{I}\|\bigvee_{\gamma\in I} |x(\gamma
)|e_{\gamma}\|\\
&  = K \sup_{I}\|S(|x(\gamma)|)_{\gamma\in I}\| \leq K\|S\|\|(x(\gamma
))\|_{\infty}.
\end{align*}
It is now clear that the map $R: E \to\ell^{\infty}(\Gamma)$ given by $Rx =
(x(\gamma))$ is a lattice isomorphism from $E$ onto a closed sublattice of
$\ell^{\infty}(\Gamma)$.
\end{proof}

Our final result shows that containment of a closed sublattice isomorphic to
$c_{0}$ is also a stable property under order isomorphisms. This holds in fact
in the category of quasi-Banach lattices. Let $E$ be a real or complex vector
space. A \emph{quasi-norm} on $E$ is a functional $\|\cdot\|$ on $E$ such that

\begin{enumerate}
\item $\|x\| > 0$ if $x\neq0$,

\item $\|ax\| = |a| \|x\|$ for any scalar $a$ and any $x\in E$,

\item There is a constant $C <\infty$ such that $\|x+y\| \leq C(\|x\|
\vee\|y\|)$ for all $x, y \in E$.
\end{enumerate}

A quasi-norm on $E$ generates a Hausdorff linear topology where the sets $\{x:
\|x\| < 1/n\}$ form a neighborhood basis at $0$. If this topology is
completely metrizable, then we say that the quasi-norm is \emph{complete} and
that $E$ is a \emph{quasi-Banach space}. A \emph{quasi-Banach lattice} is a
real vector lattice equipped with a complete quasi-norm $\|\cdot\|$ such that
$|x| \leq|y|$ implies $\|x\| \leq\|y\|$ for all $x, y\in E$. Refer to
\cite{Ka} for more information regarding quasi-Banach spaces and quasi-Banach
lattices. Given a quasi-norm $\|\cdot\|$ with associated constant $C$, it is
evident that
\[
\|\sum^{n}_{k=1}x_{k}\| \leq\max_{1\leq k\leq n-1}C^{k}\|x_{k}\| \vee
C^{n-1}\|x_{n}\|.
\]
It follows that if $(x_{k})$ is a sequence in a quasi-Banach space with
$\lim\|x_{k}\| = 0$, then there is a subsequence $(x_{k_{i}})$ such that $\sum
x_{k_{i}}$ converges. It is easy to see that the positive cone $\{x: x\geq0\}$
is a closed set in a quasi-Banach lattice; equivalently, the limit of any
positive sequence is positive. Consider the following statements.

\begin{thm}
\label{t12} Let $E$ and $F$ be order isomorphic quasi-Banach lattices. If $E$
contains a closed sublattice (nonlinearly) order isomorphic to $c_{0}$, then
$F$ contains a closed sublattice linearly lattice and topologically isomorphic
to $c_{0}$.
\end{thm}

\begin{thm}
\label{t13} Let $E$ and $F$ be order isomorphic quasi-Banach lattices. If $E$
contains a closed sublattice \emph{linearly} order isomorphic to $c_{0}$, then
$F$ contains a closed sublattice linearly lattice and topologically isomorphic
to $c_{0}$.
\end{thm}

Evidently Theorem \ref{t12} is stronger than Theorem \ref{t13}. But, in fact,
the two results are equivalent. Indeed, assume that Theorem \ref{t13} holds.
If $G$ is a quasi-Banach lattice order isomorphic to $c_{0}$, then, taking $E$
to be $c_{0}$ and $F$ to be $G$ in Theorem \ref{t13}, one concludes that $G$
contains a closed sublattice linearly order isomorphic to $c_{0}$. Thus any
$E$ that satisfies the hypothesis of Theorem \ref{t12} also fulfills the
condition of Theorem \ref{t13}.

We now proceed to prove Theorem \ref{t13} (and hence also Theorem \ref{t12}).
First observe that in order to produce a closed sublattice of $F$ linearly
order and topologically isomorphic to $c_{0}$, it suffices to obtain a
disjoint sequence $(y_{i})$ in $F$ such that $\inf\|y_{i}\| > 0$ and $\sup
_{j}\|\sum^{j}_{i=1}y_{i}\| < \infty$.

\begin{lem}
\label{l14} Let $G$ be a quasi-Banach lattice and let $S: c_{0} \to G$ be a
(linear) lattice isomorphism. Denote by $(e_{k})$ the unit vector basis of
$c_{0}$. Then $\inf\|Se_{k}\| > 0$.
\end{lem}

\begin{proof}
Otherwise, by the observation preceding Theorem \ref{t12}, there is a
subsequence $(e_{k_{i}})$ such that $x = \sum Se_{k_{i}}$ converges in $G$.
Since the positive cone of $G$ is closed, $x\geq\sum^{m}_{i=1}Se_{k_{i}} =
S(\sum^{m}_{i=1}e_{k_{i}})$ for all $m$. Then $S^{-1}x \geq\sum^{m}%
_{i=1}e_{k_{i}}$ for all $m$, which is clearly absurd.
\end{proof}

\begin{lem}
\label{l15} Let $E$ and $F$ be quasi-Banach lattices and let $T:E\to F$ be an
order isomorphism such that $T0 = 0$. If $(x_{k})$ is a disjoint sequence in
$E_{+}$ with $\inf\|x_{k}\| > 0$, then there exists $N \in{\mathbb{N}}$ such
that $\limsup\|T(Nx_{k})\| > 0$.
\end{lem}

\begin{proof}
Otherwise, there is a subsequence $(x_{k_{i}})$ such that $\lim\|T(ix_{k_{i}%
})\|= 0$. By using a further subsequence if necessary, we may assume that $y =
\sum T(ix_{k_{i}})$ converges in $F$. Since $T$ is an order isomorphism and
$T0 = 0$, $T(ix_{k_{i}}) \geq0$ for all $i$. Thus $y \geq T(ix_{k_{i}}) \geq0$
for all $i$. Hence $T^{-1}y \geq ix_{k_{i}}\geq0$ for all $i$. Therefore,
$\|T^{-1}y\| \geq i\|x_{k_{i}}\| \geq i\inf_{k}\|x_{k}\|$ for all $i$, which
is impossible.
\end{proof}

\begin{proof}
[Proof of Theorem \ref{t13}]Let $G$ be a closed sublattice of $E$ and let
$S:c_{0}\to G$ and $T:E\to F$ be order isomorphisms, where $S$ is linear and,
without loss of generality, $T0= 0$. Denote the unit vector basis of $c_{0}$
by $(e_{k})$. By Lemma \ref{l14}, $\inf\|Se_{k}\| > 0$. Let $x_{k} = Se_{k}$.
Since $(x_{k})$ is a disjoint sequence in $E_{+}$, by Lemma \ref{l15}, there
exists $N\in{\mathbb{N}}$ and an infinite subset $I$ of ${\mathbb{N}}$ so that
$\inf_{k\in I}\|T(Nx_{k})\| > 0$.

Assume that there exists $\eta> 0$ such that $\inf_{k\in I}\|T(\eta x_{k})\| =
0$. There is an increasing sequence $(k_{i})$ in $I$ such that $y = \sum
T(\eta x_{k_{i}})$ converges in $F$. Then $y \geq T(\eta x_{k_{i}})$ and hence
$T^{-1}y \geq\eta x_{k_{i}}$ for all $i$. Thus $x = (N/\eta)T^{-1}y \geq
Nx_{k_{i}}$ and so $Tx \geq T(Nx_{k_{i}})$ for all $i$. Let $y_{i} =
T(Nx_{k_{i}})$ for all $i$. Then $(y_{i})$ is a disjoint sequence in $F$ such
that $\inf\|y_{i}\| > 0$. Furthermore, $0 \leq\sum^{j}_{i=1}y_{i} =
\bigvee^{j}_{i=1}y_{i} \leq Tx$ for all $j$. Hence $\|\sum^{j}_{i=1}y_{i}\|
\leq\|Tx\|$ for all $j$. By the remark preceding Lemma \ref{l14}, $F$ has a
closed sublattice linearly order and topologically isomorphic to $c_{0}$.

Finally, suppose that $\inf_{k\in I}\|T(\eta x_{k})\| > 0$ for all $\eta> 0$.
Let $(k_{i})$ be an increasing sequence in $I$. We claim that there exists
$\varepsilon>0$ such that $\sup_{j}\|T(\varepsilon\sum^{j}_{i=1}x_{k_{i}})\| <
\infty$. Otherwise, there is an increasing sequence $(j_{m})$ such that
$\|T(2^{-m}\sum^{j_{m}}_{i=1}x_{k_{i}})\|> m$ for all $m$. The element
\[
u = \sum^{\infty}_{m=1}2^{-m}\sum^{j_{m}}_{i=1}e_{k_{i}}%
\]
belongs to $c_{0}$ and majorizes $2^{-m}\sum^{j_{m}}_{i=1}e_{k_{i}}$ for each
$m$. Since $S$ is linear and order preserving, $x = Su \geq2^{-m}\sum^{j_{m}%
}_{i=1}x_{k_{i}}$ and thus $Tx \geq T(2^{-m}\sum^{j_{m}}_{i=1}x_{k_{i}})\geq0$
for all $m$. But then $\|Tx\| > m$ for all $m$, reaching a contradiction.
Hence the claim is verified. Let $y_{i} = T(\varepsilon x_{k_{i}})$. Then
$(y_{i})$ is a disjoint sequence and $\inf\|y_{i}\| > 0$. For any $j$,
\[
0 \leq\sum^{j}_{i=1}y_{i} = \bigvee^{j}_{i=1}y_{i} = \bigvee^{j}%
_{i=1}T(\varepsilon x_{k_{i}}) = T(\bigvee^{j}_{i=1} \varepsilon x_{k_{i}}) =
T(\sum^{j}_{i=1}\varepsilon x_{k_{i}}).
\]
Therefore,
\[
\sup_{j}\|\sum^{j}_{i=1}y_{i}\| \leq\sup_{j}\|T(\sum^{j}_{i=1}\varepsilon
x_{k_{i}})\| < \infty.
\]
Again, by the remark preceding Lemma \ref{l14}, we conclude that $F$ has a
closed sublattice linearly order and topologically isomorphic to $c_{0}$.
\end{proof}

\noindent\textbf{Remark}. If $E$ is a Banach lattice, then $E$ does not
contain a closed sublattice lattice isomorphic to $c_{0}$ if and only if $E$
is weakly sequentially complete. Thus Theorem \ref{t12} shows that the
topological property of weak sequential completeness is preserved under
nonlinear order isomorphisms between Banach lattices.

\end{document}